\documentclass[a4paper,12pt,centertags,final]{amsart}
\usepackage[T1]{fontenc}
\usepackage[utf8]{inputenc}
\usepackage{mathrsfs}
\usepackage{bbm}
\usepackage{amssymb}
\usepackage{eulervm}
\usepackage{mathtools}
\usepackage{xspace}
\usepackage[a4paper,centering]{geometry}
\usepackage[pdfdisplaydoctitle,colorlinks,breaklinks,urlcolor=blue,linkcolor=blue,citecolor=blue]{hyperref}
\usepackage{mathscinet}
\def\MRnum#1\empty{#1}
\renewcommand{\MRhref}[2]{%
  \href{http://www.ams.org/mathscinet-getitem?mr=#1}{#2}
}
\renewcommand{\MR}[1]{
  \relax\ifhmode\unskip\space\fi
  \MRhref{\MRnum#1\empty}{\texttt{\Tiny[MR\MRnum#1\empty]}}
}
\newcommand{\arxiv}[1]{\href{http://arxiv.org/abs/#1}{\scriptsize arXiv: #1}}
\newcommand{\etalchar}[1]{$^{#1}$}
\newcommand{\e}{\operatorname{e}}
\newcommand{\im}{\mathrm{i}}

\newcommand{\E}{\mathbb{E}}
\newcommand{\N}{\mathbb{N}}
\newcommand{\R}{\mathbb{R}}

\newcommand{\Bc}{\mathcal{B}}
\newcommand{\Ec}{\mathcal{E}}
\newcommand{\Fc}{\mathcal{F}}
\newcommand{\Lc}{\mathcal{L}}
\renewcommand{\Mc}{\mathcal{M}}
\newcommand{\Pc}{\mathcal{P}}
\newcommand{\Es}{\mathscr{E}}
\newcommand{\Gs}{\mathscr{G}}
\newcommand{\Torus}{{\mathbb{T}_2}}
\newcommand{\uno}{\mathbbm{1}}
\newcommand{\Prob}{\mathbb{P}}
\newcommand{\eqdef}{\vcentcolon=}
\newcommand{\qedef}{=\vcentcolon}
\newcommand{\scalar}[1]{\langle #1 \rangle}
\newcommand{\memo}[1]{
  \framebox{\tiny\text{\kern-2pt\textsf{\ensuremath{#1}}}\kern-2pt}
  \xspace
}
\newtheorem{theorem}{Theorem}[section]
\newtheorem{lemma}[theorem]{Lemma}
\newtheorem{proposition}[theorem]{Proposition}

\theoremstyle{definition}

\theoremstyle{remark}
\newtheorem{remark}[theorem]{Remark}
\numberwithin{equation}{section}
\hypersetup{%
  pdftitle={Limit theorems and fluctuations for point vortices of generalized Euler equations},
  pdfauthor={C. Geldhauser, M. Romito},
  pdfcreator={M. Romito}
}
\begin{document}
  \title[Generalized models of point vortices]
    {Limit theorems and fluctuations for point vortices of generalized Euler equations}
  \author[C. Geldhauser]{Carina Geldhauser}
    \address{Chebyshev Laboratory of Sankt Petersburg State University,
      Faculty for Mathematics and Mechanics,
      Vasilievsky Island 14th line No. 29,
      Sankt Peterburg, Russia, 199178}
    \email{\href{mailto:k.geldhauser@spbu.ru}{k.geldhauser@spbu.ru}}
    \urladdr{\url{http://www.cgeldhauser.de}}
  \author[M. Romito]{Marco Romito}
    \address{Dipartimento di Matematica, Universit\`a di Pisa, Largo Bruno Pontecorvo 5, I--56127 Pisa, Italia }
    \email{\href{mailto:marco.romito@unipi.it}{marco.romito@unipi.it}}
    \urladdr{\url{http://people.dm.unipi.it/romito}}
  \date{October 30, 2018}
  \begin{abstract}
    We prove a mean field limit, a law of large numbers and
    a central limit theorem for a system of point vortices
    on the 2D torus at equilibrium with positive temperature.
    The point
    vortices are formal solutions of a class of
    equations generalising the Euler equations, and
    are also known in the literature as generalised
    inviscid SQG. The mean field limit is a steady
    solution of the equations, the CLT limit is a stationary
    distribution of the equations.
  \end{abstract}
\maketitle
\section{Introduction}

The paper analyses the mean field limit and the corresponding
fluctuations for the point vortex dynamics, at equilibrium
with positive temperature, arising from a class of equations
generalising the Euler equations. More precisely,
we consider the family of models
\[
  \partial_t\theta + u\cdot\nabla\theta
    = 0,
\]
on the two dimensional torus $\Torus$ with
periodic boundary conditions and zero spatial
average. Here $u=\nabla^\perp(-\Delta)^{-\frac{m}2}\theta$
is the velocity, and $m$ is a parameter.
When $m=2$, the model corresponds to the
Euler equations, when $m=1$ this is the
inviscid surface quasi-geostrophic
(briefly, SQG).

As in the case of Euler equations,
a family of \emph{formal} solutions is
given by point vortices, namely
measure solutions described by
\[
  \sum_{j=1}^N \gamma_j\delta_{X_j(t)},
\]
where $X_1,X_2,\dots,X_N$ are vortex positions
and $\gamma_1,\gamma_2,\dots,\gamma_N$
are vortex intensities. Positions evolve
according to \eqref{e:motion}, and intensities
are constant by a generalized version of
Kelvin's theorem. This evolution is
Hamiltonian with Hamiltonian \eqref{e:hamiltonian},
and has a family of invariant
distributions \eqref{e:mu} indexed
by a parameter $\beta$. Unfortunately
when $m<2$ the invariant distributions,
written in terms of a density which is
the exponential of the Hamiltonian,
do not make sense since
the Green function of the fractional
Laplacian $(-\Delta)^{\frac{m}2}$
has a singularity
which is too strong.

Nevertheless, the main aim of the paper
is to realize the program developed
in \cite{CagLioMarPul1992,CagLioMarPul1995,Lio1998,BodGui1999}
for point vortices for the Euler equations.
To this aim we introduce a regularization
of the Green function. At the level
of the regularized problem all
statements we are interested in (mean field
limit, analysis of fluctuations)
are not difficult, since the interaction
among vortices is bounded.
We recover the original problem
in the limit of infinite vortices,
since when the number of vortices
$N$ increases to $\infty$,
we choose the
regularization parameter
$\epsilon$ so that it goes
at the
same time to $0$. To
ensure the validity of
our result though, the
speed of convergence of
$\epsilon=\epsilon(N)$
must be at least
logarithmically
slow in terms of $N$.

Under the conditions
$\beta>0$ and $m<2$, and when
$\epsilon(N)\downarrow0$,
we prove propagation of chaos,
namely vortices decorrelate and
in the limit are independent.
A law of large numbers holds
and, in terms of $\theta$,
the limit is a stationary
solution of the original
equation. Likewise, a
central limit theorem
holds. The limit Gaussian
distribution for the
$\theta$ variable turns
out to be a statistically
stationary solution
of the equations.
\bigskip

The paper is organized as follows. In Section~\ref{s:model}
we introduce the model with full details, we give some
preliminary results and we prepare the framework to
state the main results. Section~\ref{s:main} contains
the main results, as well as some consequences and
additional remarks. Finally, Section~\ref{s:proofs}
is devoted to the proof of the main results.

We conclude this introduction with a list of
notations used throughout the paper.
We denote by $\Torus$ the two dimensional
torus, and by $\ell$ the normalized Lebesgue
measure on $\Torus$. Given a metric space $E$,
we shall denote by $C(E)$ the space of
continuous functions on $E$, and by
$\Pc(E)$ the set of probability measures
on $E$. If $x\in E$, then $\delta_x$
is the Dirac measure on $x$.
Given a measure $\mu$ on $E$,
we will denote by
$\mu(F)=\scalar{F,\mu}=\int F(x)\,\mu(dx)$
the integral of a function $F$ with respect
to $\mu$. Sometimes we will also use
the notation $\E_\mu[F]$. We will use
the operator $\otimes$ to denote the
product between measures.
We shall denote by $\lambda_1,\lambda_2,\dots$
the eigenvalues in non-decreasing order,
and by $e_1,e_2,\dots$ the corresponding
orthonormal basis of eigenvectors
of $-\Delta$, where $\Delta$ is the
Laplace operator on $\Torus$ with periodic
boundary conditions and zero spatial average.
With these positions,
if $\phi=\sum_k \phi_k e_k$, then
the fractional Laplacian is defined as
\[
  (-\Delta)^{\frac\alpha2}\phi
    = \sum_{k=1}^\infty \lambda_k^{\frac\alpha2}\phi_k e_k.
\]
\subsection*{Acknowledgments}

The authors wish to thank Franco Flandoli and Francesco Grotto
for several fruitful conversations on the subject, and
for having pointed out the paper \cite{BenPicPul1987}.

The first author acknowledges the hospitality and support
of the Mathematisches Forschungsinstitut Oberwolfach
through a Oberwolfach-Leibniz Fellowship.

The second author acknowledges the partial support
of the University of Pisa, through project PRA 2018\_49.
\section{The model}\label{s:model}

Consider the family of models,
\begin{equation}\label{e:euler}
  \partial_t\theta + u\cdot\nabla\theta
    = 0,
\end{equation}
on the torus with periodic boundary conditions
and zero spatial average, where the velocity
$u=\nabla^\perp\psi$,
and the stream function $\psi$ is
solution to the following problem,
\[
  (-\Delta)^{\frac{m}{2}}\psi
    = \theta,
\]
with periodic boundary conditions and zero spatial average.
Here $m$ is a parameter. The case $m=2$ corresponds
to the Euler equation in vorticity formulation,
$m=1$ is the inviscid surface quasi-geostrophic
equation (briefly, SQG), and for a general value
is sometimes known in the literature as the
inviscid \emph{generalized surface quasi-geostrophic}
equation. Here we will consider values $m<2$
of the parameter.
\subsection{Generalities on the model}

We start by giving a short introduction to
the main features of the model~\eqref{e:euler}.
\subsubsection{Existence and uniqueness of solution}

The inviscid SQG has been derived in meteorology to model
frontogenesis, namely the production of fronts due to
tightening of temperature gradients and is
an active subject of research. See
\cite{ConMajTab1994,HelPieSwa1994,HelPieGarSwa1995}
(see also \cite{CorFefRod2004,Rod2005}) for the
first mathematical and geophysical studies about
strong fronts.
The generalized version of the equations bridges
the cases of Euler and SQG
and it is studied to understand the mathematical
differences between the two cases.

As it regards existence, uniqueness and regularity
of solutions, a local existence result is known,
namely data with sufficient smoothness give
local in time unique solutions with the same
regularity of the initial condition, see for
instance \cite{ChaConCorGanWu2012}. Unlike
the Euler equation, it is not known if
the inviscid SQG (as well as its generalized
version) has a global solution. 
Actually, there is numerical evidence,
see \cite{CorFonManRod2005}, of
emergence of singularities in
the generalized SQG, for $m\in[1,2)$.
On the other hand see \cite{CorGomIon2017}
for classes of global solutions.
Finally, \cite{ChaConWu2011} presents
a regularity criterion for classical
solutions.

The state is different if one turns to
weak solutions. Indeed existence of
weak solutions is known since
\cite{Res1995}, see also
\cite{Mar2008}. For existence
of weak solution for the generalized
SQG model one can see 
\cite{ChaConCorGanWu2012}.
Global flows of weak solution
with a (formal) invariant measure
(corresponding to the measure
in~\eqref{e:invariant} with $\beta=0$)
as initial condition has been
provided in \cite{NahPavStaTot2017}.
\subsubsection{Invariant quantities}\label{s:invariant}

We turn to some simple (and in general
formal, but that can be made rigorous
on classical solutions) properties of the equation.
As in the case of Euler equations,
equation \eqref{e:euler} can be solved by means
of characteristics, in the sense that
if $\theta$ is solution of \eqref{e:euler}
and $u=\nabla^\perp\theta$,
\[
  \begin{cases}
    \dot{X}
      = u(t,X_t),\\
    X(0)
      = x,
  \end{cases}
\]
then, at least formally, 
\[
  \frac{d}{dt}\theta(t,X_t)
    = \partial_t \theta(t, X_t) 
      + \dot{X}_t \cdot \nabla \theta(t, X_t)
    =  (\partial_t \theta + u \cdot \nabla \theta )(t, X_t)
    =  0,
\]
therefore $\theta(t, X_t) = \theta(0, x)$. 
This formally ensures conservation of the sign and of
the magnitude ($L^\infty$ norm) of $\theta$.

It is not difficult to see that
\eqref{e:euler} admits an infinite
number of conserved quantities,
for instance of $L^p$ norms of
$\theta$.
We are especially interested in
the quantity (that in the case
$m=2$ is the enstrophy),
\[ 
  \|\theta(t)\|_{L^2}^2
    = \int_\Torus |\theta(t,x)|^2\,d\ell
\]
and in the quantity,
\[
  \int_\Torus \theta(t,x)\psi(t,x)\,d\ell
    = \|(-\Delta)^{-\frac{m}4}\theta\|_{L^2(\ell)}^2.
\]
Here, unlike the case $m=2$, this conserved
quantity is not the kinetic energy.
Formally, corresponding to these conserved
quantities, in analogy with the invariant
measures of the Euler equations
\cite{AlbCru1990}, one can consider
the invariant measures
\begin{equation}\label{e:invariant}
  \mu_{\beta,\alpha}(d\theta)
    = \frac1{Z_{\beta,\alpha}}\e^{-\beta\|(-\Delta)^{-\frac{m}4}\theta\|^2
      -\alpha\|\theta(t)\|_{L^2}^2}\,d\theta,
\end{equation}
classically interpreted as Gaussian
measures with suitable covariance
(see Remark~\ref{r:invariant}).
\subsection{The point vortex motion}

The point vortex motion is a powerful
point of view to understand some of the
phenomenological interesting properties
of solutions of the Euler equations.
Mathematical results about the general dynamics
of point vortices \cite{MarPul1994}
and about the connection with the
equations \cite{Sch1996} are classical.
The statistical mechanics approach
to the description of point vortices,
the central topic of this paper,
dates back to some of the intuitions
in the celebrated paper of Onsager
\cite{Ons1949}, and later developed
in the physical literature
\cite{JoyMon1973,FroRue1982,EyiSpo1993,Kie1993}.
Results of mean field type are obtained
in \cite{CagLioMarPul1992,CagLioMarPul1995,Lio1998}.
Mean field limit results of point vortices
with random intensities can be found in
\cite{Ner2004,Ner2005,KieWan2012}.
The analysis of fluctuations
can be found in \cite{BenPicPul1987,BodGui1999}
and in the recent \cite{GroRom2018}.

The central topic of this paper is
to give results about the mean field
limit of a system of point vortices
governed by \eqref{e:euler}.
To be more detailed, if one
considers a configuration of $N$ point vortices
located at $x_1,x_2,\dots,x_N$, with respective
intensities $\gamma_1,\gamma_2,\dots,\gamma_N$,
that is the measure
\[
  \theta(0)
    = \sum_{j=1}^N \gamma_j\delta_{X_j}
\]
as the initial condition of
\eqref{e:euler}, one can check
that, formally, the solution
evolves as a measure of the same
kind, where the ``intensities''
$\gamma_j$ remain constant (a generalized
version of Kelvin's theorem about
the conservation of circulation),
and where the vortex positions
evolve according to the system
of equations
\begin{equation}\label{e:motion}
  \begin{cases}
    \dot X_j
      = \sum_{k\neq j}\gamma_k\nabla^\perp G_m(X_j,X_k),\\
    X_j(0)
      = x_j,
  \end{cases}
  \qquad j=1,2,\dots,N,
\end{equation}
where $G_m$ is the Green function
of the operator $(-\Delta)^{\frac{m}2}$
on the torus with periodic boundary conditions
and zero spatial average. The effective connection
between the equations and the point vortex
dynamics is not yet clear and will
be discussed elsewhere. See also
\cite{FlaSaa2018}.

The motion of vortices is described
by the Hamiltonian
\begin{equation}\label{e:hamiltonian}
  H_N(\gamma^N,X^N)
    = \frac12\sum_{j\neq k}\gamma_j\gamma_k G_m(X_j,X_k),
\end{equation}
where $X^N=(X_1,X_2,\dots,X_N)$ and
$\gamma^N=(\gamma_1,\gamma_2,\dots,\gamma_N)$.

A natural invariant distribution for the
Hamiltonian dynamics \eqref{e:motion}
should be the measure
\begin{equation}\label{e:mu}
  \mu_\beta^N(dX^N)
    = \frac1{Z_\beta^N}
      \e^{-\beta H_N(X^N,\gamma^N)}d\ell^{\otimes N},
\end{equation}
where here and throughout the paper we denote
by $\ell$ the normalized Lebesgue measure
on $\Torus$.
Due to the singularity of the Green function
on the diagonal,
which is of order $G_m(x,y)\sim|x-y|^{m-2}$,
the density above is not integrable and thus
the measure $\mu^N_\beta$ does not make sense.
\subsection{The regularized system}

To overcome this difficulty, we consider
a regularization of the Green function.
To define the regularization,
notice that we can represent the Green
function for the fractional Laplacian
through the eigenvectors,
\[
  G_m(x,y)
    = \sum_{k=1}^\infty \lambda_k^{-\frac{m}2}e_k(x)e_k(y).
\]
Given $\epsilon>0$, consider the following regularization
of the Green function,
\begin{equation}\label{eq:greenreg}
  G_{m,\epsilon}(x,y)
    = \sum_{k=1}^\infty \lambda_k^{-\frac{m}{2}}\e^{-\epsilon\lambda_k}e_k(x)e_k(y).
\end{equation}
Here, we have regularized the fractional Laplacian
so that the new operator $D_{m,\epsilon}$ reads
$D_{m,\epsilon}=(-\Delta)^{m/2}\e^{-\epsilon \Delta}$
and the eigenvalues change from $\lambda^{m/2}$ to
$\lambda^{m/2}\e^{\epsilon\lambda}$.
We remark that, as long as $G_{m,\epsilon}$ is translation
invariant and non-singular on the diagonal, the exact form
of the regularization is not essential for our main results
given in Section~\ref{s:main}.

If we replace $G_m$ by $G_{m,\epsilon}$ in \eqref{e:motion},
the motion is still Hamiltonian with Hamiltonian
$H_N^\epsilon$ given
by \eqref{e:hamiltonian}, with $G_m$ replaced by
$G_{m,\epsilon}$, namely
\[
  H_N^\epsilon(\gamma^N,X^N)
    = \frac12\sum_{j\neq k}\gamma_j\gamma_k G_m^\epsilon(X_j,X_k).
\]
In terms of invariant distributions,
we want to consider a problem slightly more general and
we shall randomize the intensities of vortices.
The ``quenched'' case, namely the case with fixed
intensities, will follow as a by-product, see
Remark~\ref{r:quenched}.

Let $\nu$ be a probability measure on
the real line with support on a compact
set $K_\nu\subset\R$\footnote{In other
  words we assume that intensities are
  bounded in size by a deterministic constant.
  Notice that in the case $m=2$ this
  is in a way a requirement,
  see \cite{CagLioMarPul1992,CagLioMarPul1995,Lio1998}.}.
The measure $\nu$ will be the \emph{prior} distribution
on vortex intensities. A natural invariant distribution
for the regularized Hamiltonian dynamics
with random intensities is
\begin{equation}\label{e:reg_mu}
  \mu_{\beta,\epsilon}^N(d\gamma^N,dX^N)
    = \frac1{Z_{\beta,\epsilon}^N}
      \e^{-\frac{\beta}{N} H_N^\epsilon(\gamma^N,X^N)}
      \,d\ell^{\otimes N}\,d\nu^{\otimes N},
\end{equation}
where $\ell$ is the normalized Lebesgue measure
on $\Torus$ and $Z_{\beta,\epsilon}^N$ is the
normalization factor. In the above formula for the measure
we have scaled the parameter $\beta$ by $N^{-1}$,
in analogy with the case $m=2$.
Indeed, for the Euler equation
there is no nontrivial thermodynamic limit
\cite{FroRue1982}, and the interesting
regime that provides interesting
results is the mean field limit.
See \cite{MarPul1994} for a physical motivation,
and \cite{CagLioMarPul1992,CagLioMarPul1995,Lio1998}
for the related mathematical results.
\subsubsection{Mean field limit of the regularized system}\label{s:reg_mean_field}

The problem of finding the limit of measures
$(\mu_{\beta,\epsilon}^N)_{N\geq1}$ is
trivial, since the interaction among particles
is bounded, and we only give an outline of
the results.

The goal of this section is to show the existence
of limit points for the measure
\eqref{e:reg_mu} and to characterize them. 
We look at convergence of the distributions
of a finite number of vortices.
Therefore, we work with the so-called
\emph{correlation functions}, defined by
\[
  \rho_{\beta,\epsilon}^{N,k}(\gamma^k,X^k)
    \eqdef\int_{\Torus^{N-k}}\mu_{\beta,\epsilon}^N(d\gamma^{N-k},dX^{N-k}),
\]
namely the distribution of the first $k$ vortices.
This is not restrictive, by exchangeability
of the measures~\eqref{e:reg_mu}.

In the following result we summarise the relevant
estimates and therefore deduce weak convergence
of the correlation functions.
\begin{lemma}
  There is a number $C>0$ that depends (only)
  on $\beta$ and $\epsilon$, but not on $k\geq1$
  and $N\geq1$, such that the following bounds hold,
  \[
    \begin{aligned}
      &Z^N_{\beta,\epsilon}
        \leq C^N,\\
      &\rho_{\beta,\epsilon}^{N,k}
        \leq C^k\e^{-\frac{\beta}{N}H_k^\epsilon(\gamma^k,X^k)},\\
      &\|\rho_{\beta,\epsilon}^{N,k}\|_{L^p}
        \leq C^k,
          \qquad p\in[1,\infty).
    \end{aligned}
  \]
  In particular, there is a sub-sequence $(N_j)_{j\geq1}$
  such that
  \[
    \rho_{\beta,\epsilon}^{N_j,k}
      \rightharpoonup \rho_{\beta,\epsilon}^k
  \]
  weakly in $L^p((K_\nu\times\Torus)^k)$, for all
  $k\geq1$ and $p\in[1,\infty)$.
\end{lemma}
The proof of this lemma is a simpler version, due
to the boundedness of the Green function, of
corresponding results from \cite{Ner2004}, and is
therefore omitted. 

To characterize the limit, consider the
free energy functional on
measures on $(K_\nu\otimes\Torus)^N$,
\[
  \Fc_N^\epsilon(\mu)=
    \Ec(\mu|\nu^{\otimes N}\otimes\ell^{\otimes N})
      + \frac{\beta}{N}\int H_N^\epsilon(\gamma^N,X^N)\mu(d\gamma^N,dX^N)
\]
where $\Ec$ is the relative entropy. It is not difficult to see that
$\mu_{\beta,\epsilon}^N$ is the unique minimiser of the
free energy. This can be carried to the limit. By
convexity and subadditivity, we can define the
entropy
\[
  \Ec_\star(\mu)
    = \lim_{N\to\infty}\frac1N\Ec(\rho^N|\nu^{\otimes N}\otimes\ell^{\otimes N}).
\]
and thus the limit free energy as
\begin{equation}\label{e:freenergy}
  \Fc_\star^\epsilon(\mu)
    = \Ec_\star(\mu)
      + \frac12\beta\iint H_2^\epsilon(\gamma^2,X^2)\rho^2(d\gamma^2,dX^2),
\end{equation}
where $(\rho^N)_{N\geq1}$ are the correlation functions of $\mu$.
Here $\Ec_\star$ and $\Fc_\star^\epsilon$ are defined on exchangeable
measures on $(K_\nu\times\Torus)^\N$ with absolutely continuous
(with respect to powers of $\nu\otimes\ell$) correlation measures,
with bounded densities. 

As in \cite[Theorem 11]{Ner2004}, we have the following result.
\begin{proposition}[Propagation of chaos]
  All limit points of $(\mu_{\beta,\epsilon}^N)_{N\geq1}$ are
  minima of the functional $\Fc_\star^\epsilon$.
  
  If $\Fc_\star^\epsilon$ has a unique minimum $\mu$,
  then $\mu$ is a product measure, namely there is
  a bounded function $\rho$
  such that all correlation functions $(\rho^k_\mu)_{k\geq1}$
  of $\mu$ have densities
  \[
    \rho^k_\mu(\gamma^k,x^k)
      = \rho(\gamma_1,x_1)\cdot\rho(\gamma_2,x_2)
        \cdot\dots\cdot\rho(\gamma_k,x_k).
  \]
  In other words, propagation of chaos holds.
\end{proposition}
Since the measures $\mu_{\beta,\epsilon}^N$ are
symmetric, each limit point $\mu_{\beta,\infty}^\epsilon$
will be
exchangeable thus, by the De Finetti theorem,
will be a superposition of product measures,
namely
\[
  \mu_{\beta,\epsilon}^\infty
    = \int \mu^{\otimes\N}\pi(d\mu),  
\]
for a measure $\pi$ on probability measures
on $K_\nu\times\Torus$. As in
\cite[Theorem 13]{Ner2004}, each
measure $\pi$ is concentrated on
product measures $\mu^{\otimes\N}$
such that $\mu=\rho\,\nu\otimes\ell$
and the variational principle
for $\Fc_\star^\epsilon$ can be read as a variational
principle for $\rho$, in terms of
the free energy
\begin{equation}\label{e:free2}
  \Fc^\epsilon(\rho)
    = \int \rho\log\rho\,d\nu\,d\ell
      + \frac12\beta\iint H_2^\epsilon(\gamma^2,x^2)
      \rho(\gamma_1,x_1)\rho(\gamma_2,x_2)
      \,d\nu^{\otimes2}\,d\ell^{\otimes2}.
\end{equation}
The corresponding Euler-Lagrange equation,
mean field equation in the language of
\cite{CagLioMarPul1992,CagLioMarPul1995,Lio1998}, is
\begin{equation}\label{e:mfe}
  \rho(\gamma,x)
    = \frac1Z\e^{-\beta\gamma\psi_\rho(x)},
\end{equation}
where $Z$ is the normalization constant,
and $\psi_\rho$ is the averaged stream
function, that is
$\psi_\rho(x)=\int\gamma G_{m,\epsilon}(x,y)\rho(\gamma,y)\,d\nu\,d\ell$.
It is elementary to check that the function
$\rho_0=1$ is a solution, with stream function
$\psi_{\rho_0}=0$. If $\mu_0=(\rho_0\nu\otimes\ell)^\N$
is the product measure corresponding to $\rho_0$,
then it is easy to check that
$\Fc_\star^\epsilon(\mu_0)=0$.
If $\beta\geq0$, the limit free energy
$\Fc_\star^\epsilon$ is non--negative, and
$\mu_0$ is the unique minimum.
As in \cite{BodGui1999}, one
can actually show that
there is only one minimiser for
small negative values of $\beta$,
and thus propagation of chaos
also holds for those values of $\beta$
and limit measure $\mu_0$.
\subsubsection{Back to the original problem}

The program outlined here for $(\mu_{\beta,\epsilon}^N)_{N\geq1}$
does not work at $\epsilon=0$ from the very beginning,
because, as already pointed out, the densities
are too singular. On the other hand
the limit free energy $\Fc_\star^\epsilon$
makes sense at $\epsilon=0$, as well
as the mean field equation
\eqref{e:mfe}.
Moreover, as long as $\Fc_\star^0$ is
convex and non--negative, the unique
minimum is again $\mu_0$.
In Section~\ref{s:main} we prove that,
by taking the limit of measures
$(\mu_{\beta,\epsilon_N}^N)_{N\geq1}$,
with a careful choice of the sequence
$\epsilon_N\downarrow0$, one can derive,
at least when $\beta\geq0$, propagation of
chaos, a law of large numbers and
a central limit theorem for the
empirical density of the pair
intensity-position of vortices.

Before that
we wish to give some comments about
the case when $\beta$ is negative.
First of all, we do not expect
that $\nu\otimes\ell$ will
be a minimiser for all negative
values of $\beta$. This is true
for all values of $\epsilon$,
in particular for the interesting
case $\epsilon=0$ and this is the
reason we give a detailed computation
below. The computation is similar to
\cite[section 5.3]{Lio1998}.
\begin{lemma}
  Let $\epsilon\geq0$ and $\beta<0$. Then
  $\mu_0=\nu\otimes\ell$ is not a minimiser
  of the free energy \eqref{e:free2}
  for $\beta<\beta_0$, where
  \[
    \beta_0
      \eqdef - \frac{\lambda_1^{\frac{m}2}\e^{\epsilon\lambda_1}}
        {\nu(\gamma^2)},
  \]
  and $\nu(\gamma^2)=\int\gamma^2\,d\nu$.
\end{lemma}
\begin{proof}
  Let $\varphi$ be bounded and with zero average with
  respect to $\nu\otimes\ell$, and set $\rho_t=1+t\varphi$,
  so that $\rho_t\nu\otimes\ell$ is a perturbation
  of $\mu_0$ for $t$ small. Clearly $\Fc^\epsilon(\rho_0)=0$
  and
  \[
    \Fc^\epsilon(\rho_t)
      = \int\rho_t\log\rho_t\,d\nu\,d\ell
      + \frac12\beta t^2\|(-\Delta)^{-\frac{m}4}
      \e^{\frac12\epsilon\Delta}\bar\varphi\|_{L^2(\ell)}^2
  \]
  where $\bar\varphi(x)=\int\gamma\varphi(\gamma,x)\,\nu(d\gamma)$.
  Expand the entropy around $t=0$ and choose
  $\varphi=\gamma e_1$, to get
  \[
    \Fc^\epsilon(\rho_t)
      = \Fc^\epsilon(\rho_0)
      + \frac12t^2\bigl(1+\beta\lambda_1^{-\frac{m}2}
        \e^{-\epsilon\lambda_1}\nu(\gamma^2)\bigr)
      + o(t^2).
  \]
  With the choice of $\beta$ as in the statement,
  $\mu_0$ cannot be a minimiser.
\end{proof}
The previous result can be read in terms of
the equation for the averaged stream function.
If $\rho$ is a solution of the mean field
equation \eqref{e:mfe}, define the
averaged scalar,
\[
  \theta_\rho(x)
    = \int\gamma\rho(\gamma,x)\nu(d\gamma)
      - \iint\gamma\rho(\gamma,x)\,d\nu\,d\ell,
\]
and the averaged stream function
$\psi_\rho=\int G_{m,\epsilon}(x,y)\theta_\rho(y)\,dy$,
then $\psi_\rho$ satisfies the following version
of \eqref{e:mfe},
\[
  D_{m,\epsilon}\psi
    = \frac{\int\gamma\e^{-\beta\gamma\psi_\rho}\,d\nu
      - \iint\gamma\e^{-\beta\gamma\psi_\rho}\,d\nu\,d\ell}
      {\iint\e^{-\beta\gamma\psi_\rho}\,d\nu\,d\ell}.
\]
where $D_{m,\epsilon}=(-\Delta)^{\frac{m}2}\e^{-\epsilon\Delta}$.
For $\rho=1$, $\theta_\rho=\psi_\rho=0$.
The linearisation of the above nonlinear equation
around $\psi=0$ yields the operator
$D_{m,\epsilon}+\beta\nu(\gamma^2)I$,
which is positive definite for
$\epsilon\geq0$ and $\beta$ as in
the previous lemma. At least
when $\epsilon>0$, due to uniform
bounds on the minima that one can derive as
in \cite[Property 2.2]{BodGui1999},
this shows that the previous lemma is optimal.
In the case $\epsilon=0$ unfortunately
these bounds are not available and
this is only an indication on what could happen.

We do not know if a law of large numbers
and a central limit theorem hold
for $-\beta_0<\beta<0$,
or if Gaussian fluctuations hold up 
to the value $\beta_0$. This is the
subject of a work in progress.
\begin{remark}
  One can derive a large deviation principle,
  as in \cite{BodGui1999}, for the regularized
  system at fixed $\epsilon>0$. We have not been
  able to derive a large deviation principle
  in the limit $N\uparrow\infty$ and
  $\epsilon=\epsilon(N)\downarrow0$, similarly to
  the results of the next section, due to
  a unsatisfactory control of the free energy,
  the expected rate function.
\end{remark}
\section{Main results}\label{s:main}

In this section we illustrate our main results, that is convergence
of distributions of a finite number of vortices and propagation
of chaos, and a law of large numbers and a central limit theorem
for the point vortex system under the assumption of positive
temperature $\beta>0$. Our results are asymptotic both in
the number of vortices \emph{and} the regularization
parameter $\epsilon$, and thus they capture the behaviour
of the original system \eqref{e:euler}. The results hold,
though, only if the regularization parameter is allowed
to go to zero with a speed, with respect to the number
of vortices, which is at least \emph{logarithmically slow}.

We know from Section~\ref{s:reg_mean_field}
that, at finite $\epsilon$, propagation of chaos holds and
the limit distribution of a pair (position, intensity) is
the measure $\nu\otimes\ell$. This is also the candidate
limit when $\epsilon,N$ converge jointly to $0$ and $\infty$.
This is the first main result of this section.
\begin{theorem}[Convergence of finite dimensional distributions]\label{t:chaos}
  Assume $m<2$ and $\beta>0$, and fix a sequence
  $\epsilon=\epsilon(N)\downarrow0$ so that
  \begin{equation}\label{e:anneal}
    \epsilon(N)
      \downarrow0
        \quad\text{as}\quad N\uparrow\infty,\qquad
    \epsilon(N)\geq C(\log N)^{-\frac2{2-m}}
  \end{equation}
  with $C$ large enough (depending on $\nu$ and $\beta$).
  Then, as $N\to \infty$, the $k$-finite dimensional marginals of $\mu_{\beta,\epsilon}^N$ converge to $(\nu \otimes \ell)^{\otimes k}$.
  In particular, propagation of chaos holds.
\end{theorem}
The proof of convergence of finite dimensional distribution
will be given in Section~\ref{s:chaos}.

We turn to the second limit theorem. Consider a system of $N$ point
vortices at equilibrium, with equilibrium measure~\eqref{e:reg_mu},
described by the $N$ pairs
$(\gamma_1^N,X_1^N),(\gamma_2^N,X_2^N),\dots,(\gamma_N^N,X_N^N)$
of intensity and position. Define the joint empirical distribution
\[
  \eta_N
    = \frac1N\sum_{j=1}^N \delta_{(\gamma_j^N,X_j^N)}
\]
of intensity and position of point vortices.
\begin{theorem}[Law of large numbers]\label{t:lln}
  Assume $m<2$ and $\beta>0$, and choose $\epsilon=\epsilon(N)$ as in
  the previous theorem. Then 
  \[
    \eta_N
      \rightharpoonup\nu\otimes\ell,
        \qquad\emph{in probability}
  \]
  as $N\uparrow\infty$.
\end{theorem}
The proof of the law of large numbers is postponed to
Section~\ref{s:lln}.
\begin{remark}\label{r:pseudolln}
  It is elementary to verify that
  convergence of $\eta_N$ to $\ell\otimes\nu$ implies
  immediately convergence of the empirical
  pseudo-vorticity,
  \[
    \theta_N
      =\frac1N\sum_{j=1}^N\gamma_j^N\delta_{X_j^N}
  \]
  to $\nu(\gamma)\ell$, with $\nu(\gamma)=\int\gamma\,\nu(d\gamma)$.
  This yields a law of large numbers for the
  empirical pseudo-vorticity.
\end{remark}
Finally, we can analyze fluctuations with respect to
the limit stated in the previous theorem, namely
the limit of the measures
\[
  \zeta_N
    = \sqrt{N}(\eta_N - \nu\otimes\ell)
\]
to a Gaussian distribution.
To this end define the operators $\Es$, $\Gs$ as
\[
  \begin{gathered}
    \Gs\phi(x)
      \eqdef\int_\Torus G_m(x,y)\phi(y)\,\ell(dy),\\
    \Es\phi(\gamma,x)
      \eqdef\gamma\int_{K_\nu}\int_\Torus
        \gamma'G_m(x,y)\phi(\gamma',y)\,\nu(d\gamma')\ell(dy).
  \end{gathered}
\]
The operator $\Gs$ provides the solution to the problem
$(-\Delta)^{\frac{m}2}\Phi=\phi$ with periodic boundary conditions
and zero spatial average, and extends naturally to functions depending
on both variables $\gamma$, $x$ by acting on the spatial
variable only.
The proof of the following theorem will be
the subject of Section~\ref{s:clt}.
\begin{theorem}[Central limit theorem]\label{t:clt}
  Assume $\beta>0$ and choose $\epsilon=\epsilon(N)$ as in
  \eqref{e:anneal}. Then $(\zeta_N)_{N\geq1}$ converges,
  as $N\uparrow\infty$,
  to a Gaussian distribution with covariance
  $I - \beta(I+\beta\Gamma_\infty\Gs)^{-1}\Es$, in the
  sense that for every test function $\psi\in L^2(\nu\otimes\ell)$,
  $\scalar{\psi,\zeta_N}$ converges in law to a real centred Gaussian
  random variable with variance
  \[
    \sigma_\infty(\psi)^2
      \eqdef\scalar{I - \beta(I+\beta\Gamma_\infty\Gs)^{-1}\Es(\psi-\bar\psi),
        (\psi-\bar\psi)},
  \]
  where $\bar\psi=(\nu\otimes\ell)(\psi)$
  and $\Gamma_\infty=\nu(\gamma^2)$.
\end{theorem}
\begin{remark}\label{r:invariant}
  As in Remark~\ref{r:pseudolln}, we can derive a central limit
  theorem for the empirical pseudo-vorticity $\theta_N$
  Indeed, $\sqrt{N}(\theta_N-\nu(\gamma)\ell)$ converges to
  a Gaussian distribution with covariance
  $\Gamma_\infty(I+\beta\Gamma_\infty\Gs)^{-1}$, in the
  sense that for every test function $\psi\in L^2(\ell)$,
  $\scalar{\sqrt{N}(\theta_N-\nu(\gamma)\ell),\psi}$
  converges in law to a real centred Gaussian random
  variable with variance
  \[
    \tilde\sigma_\infty(\psi)^2
      = \Gamma_\infty\scalar{(I+\beta\Gamma_\infty\Gs)^{-1}
        (\psi-\bar{\psi}),(\psi-\bar{\psi})}.
  \]
  The Gaussian measure obtained corresponds to the invariant
  measure \eqref{e:invariant} of the original system \eqref{e:euler},
  when one takes $\alpha=1/\Gamma_\infty$.
\end{remark}
\begin{remark}[Quenched results]\label{r:quenched}
  The above results hold also in a ``quenched'' version, namely
  if intensities are non-random but given at every $N$.
  For instance, consider the result about convergence
  of finite dimensional distributions of vortices and
  propagation of chaos (Theorem~\ref{t:chaos}).
  For every $N$, fix a family
  $\Gamma_N^q\eqdef(\gamma_j^N)_{j=1,2,\dots,N}$
  and consider the quenched version of \eqref{e:reg_mu},
  \[
    \mu_{\beta,\epsilon}^{\Gamma_N^q,N}(dx_1,\dots,dx_N)
      = \frac1{Z_{\beta,\epsilon}^{\Gamma_N^q,N}}
        \e^{-\frac\beta{N}
        H_N^\epsilon(\gamma_1^N,\dots,\gamma_N^N,x_1,\dots,x_N)}
        \,d\ell^{\otimes N}.
  \]
  If there is a measure $\nu_\star$ such that
  \begin{equation}\label{e:quenchedconv}
    \frac1N\sum_{j=1}^N\delta_{\gamma_j^N}
      \rightharpoonup\nu_\star,
        \qquad N\uparrow\infty,
  \end{equation}
  and, due to our singular setting (in view of
  Lemma~\ref{l:bernstein}), if
  \[
    \Bigl|\frac1N\sum_{j=1}^N (\gamma_j^N)^2
        - \int\gamma^2\,\nu(d\gamma)\Bigr|
        G_{m,\epsilon_N}(0,0)
      \longrightarrow 0
        \qquad N\uparrow\infty,
  \]
  then the $k$-dimensional marginals of
  $\mu_{\beta,\epsilon}^{\Gamma_N^q,N}$
  converge to $(\nu\otimes\ell)^{\otimes k}$,
  for all $k\geq1$. Under the same assumptions,
  the law of large numbers also holds. To
  obtain the central limit theorem, one needs
  to assume some concentration condition
  on the convergence~\eqref{e:quenchedconv}.
\end{remark}
\section{Proofs of the main results}\label{s:proofs}

Prior to the proof of our main results we state some
preliminary results that will be useful in the rest
of the section.
\begin{lemma}\label{l:approx}
  Let $f\in L^3(\Torus)$ with zero average on $\Torus$, then
  \[
    \Big|\int_\Torus \e^{\im f(x)}\,d\ell - \e^{-\frac12\|f\|_{L^2}^2}\Big|
      \leq \|f\|_{L^3}^3.
  \]
  Here the norms $\|\cdot\|_{L^2}$ and $\|\cdot\|_{L^3}$
  are computed with respect to the normalized Lebesgue
  measure $\ell$ on $\Torus$.
\end{lemma}
\begin{proof}
  Using the well-known inequalities
  \[
  \begin{gathered}
  |\e^{\im x} - (1+\im x-\tfrac12 x^2)|
  \leq |x|^3,\\
  |\e^{-\frac12 x^2} - (1 - \tfrac12 x^2)|
  \leq |x|^3,
  \end{gathered}
  \]
  the proof is elementary.
\end{proof}
\begin{lemma}\label{l:weak_conv}
  Let $(\mu_N)_{N\geq1}$, $\mu_\infty$ be random probability
  measures on $\Torus\times K_\nu$. Then $(\mu_N)_{N\geq1}$
  converges in law to $\mu_\infty$ if and only if
  for every $\psi\in C(K_\nu\times\Torus)$,
  \[
    \E\bigl[\e^{\im\scalar{\psi,\mu_N}}\bigr]
      \longrightarrow \E\bigl[\e^{\im\scalar{\psi,\mu_\infty}}\bigr]
  \]
  Moreover, test functions can be taken
  in $C^1(K_\nu\times\Torus)$.
\end{lemma}
\begin{proof}
  Set $E=\Pc(K_\nu\times\Torus)$ and recall that
  $K_\nu\times\Torus$ is a complete compact metric space,
  therefore $E$ and $\Pc(E)$ are complete compact
  (thus separable) metric spaces for the topology
  of weak convergence.
  
  For every $\psi\in C(K_\nu\times\Torus)$ define
  $\Phi_\psi\in C(E)$ as $\Phi_\psi(\mu)=\int\psi\,d\mu$.
  Consider the subset
  \[
    \Mc
      = \{\e^{\im\Phi_\psi}:\psi\in C(K_\nu\times\Torus)\}
  \]
  of $C(E)$. By Lemma 4.3 and Theorem 4.5
  of \cite{EthKur1986}, it is sufficient to prove 
  that $\Mc$ is an algebra that separates the points of $E$.
  It is straightforward to check that $\Mc$ is an algebra.
  To prove that $\Mc$ separates points, consider
  $\mu,\nu\in E$ with $\Psi(\mu)=\Psi(\nu)$ for all
  $\Psi\in\Mc$. This reads
  \[
    \e^{\im\scalar{\psi,\mu}}
      = \e^{\im\scalar{\psi,\nu}}
  \]
  for all $\psi\in C(K_\nu\times\Torus)$. This readily
  implies that $\mu=\nu$. Likewise if
  $\psi\in C^1(K_\nu\times\Torus)$.
\end{proof}
In the proof of our limit theorems we will
streamline and adapt to our setting an idea
from \cite{BenPicPul1987}. The key point is
to give a representation of the equilibrium
measure density in terms of a Gaussian random
field. Here the condition $\beta>0$ is crucial.
\begin{lemma}\label{l:gaussian_rep}
  Let $(x_1,x_2,\dots,x_N)\in \Torus^N$ be $N$ distinct points,
  and let $\gamma_1,\gamma_2,\dots,\gamma_N\in K_\nu$.
  Then
  \[
    \e^{-\frac\beta{N}H_N^\epsilon(x^N,\gamma^N)}
      = \E_{U_{\beta,\epsilon}}\bigl[\e^{\frac\im{\sqrt N}
        \sum_{j=1}^N\gamma_j U_{\beta,\epsilon}(x_j)}\bigr]
        \e^{\frac1{2N}\beta G_{m,\epsilon}(0,0)
          \sum_{j=1}^N \gamma_j^2},
  \]
  where $U_{\beta,\epsilon}$ is the periodic mean zero Gaussian
  random field on the torus  with covariance
  $\beta G_{m,\epsilon}$, and $\E_{U_{\beta,\epsilon}}$
  denotes expectation with respect to the probability
  framework on which $U_{\beta,\epsilon}$ is
  defined.
\end{lemma}
\begin{proof}
  The proof is elementary, since by definition
  of the random field $U_{\beta,\epsilon}$,
  the random vector $(U_{\beta,\epsilon}(x_1),
  U_{\beta,\epsilon}(x_2),\dots,U_{\beta,\epsilon}(x_N))$
  is centred Gaussian with covariance matrix
  $(\beta G_{m,\epsilon}(x_j,x_k))_{j,k=1,2,\dots,N}$.
  Notice finally that by translation invariance,
  $G_{m,\epsilon}(x,x)=G_{m,\epsilon}(0,0)$.
\end{proof}
\begin{lemma}\label{l:bernstein}
  Assume there are a sequence of {i.\,i.\,d.} real random variables
  $(X_k)_{k\geq1}$ such that there is $M>0$ with $0\leq X_k\leq M$
  for all $k$, and a sequence of complex random variables
  $(Y_k)_{k\geq1}$ such that $\E Y_k\to L$, {a.\,s.} and
  $|Y_k|\leq M$ for all $k$.
  Set $S_n=\frac1n\sum_{k=1}^n X_k$, $S=\E[X_1]$.
  
  If $F_n:[-S,M]\to\R$ is a sequence of functions such that
  there is $\alpha<\frac14$ with
  \begin{itemize}
    \item $1=F_n(0)\leq F_n(y)\leq \e^{c_0 n^{2\alpha}}$ for all $y\in[-S,M]$,
    \item $\Bc_\delta\eqdef\sup_{|y|\leq\delta,n\geq1} F_n(n^{-\alpha}y)
      \longrightarrow1$ as $\delta\to0$,
  \end{itemize}
  then
  \[
    \E[F_n(S_n-S)Y_n]
      \longrightarrow L,
  \]
  as $n\to\infty$.
\end{lemma}
\begin{proof}
  Choose $\beta$ such that $\alpha\leq\beta<\frac12(1-2\alpha)$,
  fix $\delta>0$ and set
  \[
    A_n
      \eqdef\{n^\beta|S_n-S|\leq\delta\}.
  \]
  By the Bernstein inequality there is $c_1>0$ such that
  \begin{equation}\label{e:bernstein}
    \Prob[A_n^c]
      \leq \e^{-c_1 n^{1-2\beta}}.
  \end{equation}
  In particular, $n^\beta(S_n-S)\to0$ {a.\,s.}. Now,
  \[
    \E[F_n(S_n-S)Y_n]
  = \E[F_n(S_n-S)Y_n\uno_{A_n}] + \E[F_n(S_n-S)Y_n\uno_{A_n^c}]
  \qedef \memo{i} + \memo{o}.
  \]
  First, using the first assumption on $F_n$ and \eqref{e:bernstein},
  \[
    \memo{o}
      \leq M\e^{c_0 n^{2\alpha}}\Prob[A_n^c]
      \leq M\e^{c_0 n^{2\alpha}-c_1 n^{1-2\beta}}
      \longrightarrow0,
  \]
  by the choice of $\beta$. For the other term, let
  $\theta_\delta(y)=(y\wedge\delta)\vee(-\delta)$, then
  (recall that $\alpha\leq\beta$),
  \[
  \begin{aligned}
    \memo{i}
      &= \E[F_n(n^{-\alpha}\theta_\delta(n^\alpha(S_n-S)))Y_n\uno_{A_n}]\\
      &= \E\bigl[\bigl(F_n(n^{-\alpha}\theta_\delta(n^\alpha(S_n-S)))-1\bigr)
        Y_n\uno_{A_n}\bigr]
        + \E[Y_n\uno_{A_n^c}].
  \end{aligned}
  \]
  By \eqref{e:bernstein}, $\E[Y_n\uno_{A_n^c}]\to L$, moreover,
  \[
    \bigl|\E\bigl[\bigl(F_n(n^{-\alpha}\theta_\delta(n^\alpha(S_n-S)))-1\bigr)
        Y_n\uno_{A_n}\bigr]\bigr|
      \leq M(\Bc_\delta-1)
  \]
  and $\Bc_\delta\to1$ as $\delta\to0$ by the second
  assumption. The conclusion follows by first taking the limit
  in $n$, and then the limit in $\delta$.
\end{proof}
\subsection{Proof of Theorem~\ref{t:chaos}}\label{s:chaos}

This section contains the proof of convergence
of finite dimensional distributions of
the equilibrium measure \eqref{e:reg_mu}.
To this end it is sufficient to prove
convergence of the characteristic functions.

Fix $n\geq1$, and assume $N\gg n$.
By exchangeability of the measure
$\mu^N_{\beta,\epsilon}$, it is sufficient to
focus on the first $n$ vortices
$(\gamma_1,X_1),\dots,(\gamma_n,X_n)$.
We have dropped here for simplicity the
superscript ${}^N$.
Fix $a=(a_1,a_2,\dots,a_n)\in\R^k$
and $b=(b_1,b_2,\dots,b_n)\in(\R^2)^k$,
we will write
$a\cdot\gamma$ as a shorthand
for $a_1\gamma_1+\dots+a_n\gamma_n$,
as well as $b\cdot X$ as a shorthand
for $b_1\cdot X_1+\dots+b_n\cdot X_n$.

With these positions,
\[
  \E_{\mu_{\beta,\epsilon}^N}\bigl[\e^{\im(a\cdot\gamma+b\cdot X)}\bigr]
    = \frac{1}{Z_{\beta,\epsilon}^N}\int\ldots\int
      \e^{\im(a\cdot\gamma + b\cdot x)}
      \e^{-\frac\beta{N}H_N^\epsilon(\gamma,x)}
        \,d\ell^{\otimes N}\,d\nu^{\otimes N}.
\]
By using Lemma~\ref{l:gaussian_rep}, the formula above
can be re-written as
\[
  \begin{multlined}
    \E_{\mu_{\beta,\epsilon}^N}\bigl[\e^{\im(a\cdot\gamma+b\cdot X)}\bigr]=\\
      = \frac1{Z_{\beta,\epsilon}^N}\int\dots\int
        \E_{U_{\beta,\epsilon}^N}\bigl[\e^{\im(a\cdot\gamma+b\cdot x)
          + \frac{\im}{\sqrt N}\sum_{j=1}^N\gamma_j U_{\beta,\epsilon}(x_j)}\bigr]
          \e^{\frac12\beta\Gamma_N G_{m,\epsilon}(0,0)}
          \,d\ell^{\otimes N}\,d\nu^{\otimes N},
  \end{multlined}
\]
where
\begin{equation}\label{e:gammaN}
  \Gamma_N
    \eqdef\frac1N\sum_{j=1}^N \gamma_j^2,
\end{equation}
that by the law of large numbers converges
in probability to $\E_\nu[\gamma^2]=\Gamma_\infty$.

We write now the space integral of the expectation in the integral
above in a more compact way, to make our computations easier.
Extend the vector $b$ to $0$, in the sense that $b_j=0$
if $j\geq n+1$, and set 
  \[
    \begin{aligned}
      A_{\epsilon j}^N(b)
        &\eqdef\int_\Torus \e^{\im \bigl(b_j\cdot x_j
          + \frac{\gamma_j}{\sqrt N} U_{\beta\epsilon}(x_j)\bigr)}\,d\ell,\\
      B_{\epsilon j}^N(b)
        &\eqdef \e^{-\frac{\gamma_j^2}{2N}\bigl\|
          U_{\beta\epsilon}\bigr\|_{L^2(\ell)}^2}
          \int_\Torus\e^{\im b_j\cdot x_j}\,d\ell,\\
      D_{\epsilon j}^N(b)
        &\eqdef A_{\epsilon j}^N(b) - B_{\epsilon j}^N(b).
\end{aligned}
\]
We thus have
\[
    \E_{\mu_{\beta,\epsilon}^N}\bigl[\e^{\im(a\cdot\gamma+b\cdot X)}\bigr]
      = \frac1{Z_{\beta,\epsilon}^N}\int\ldots\int
        \e^{\im a\cdot\gamma}\E_{U_{\beta,\epsilon}^N}\Bigl[
        \prod_{j=1}^N A_{\epsilon j}^N(b)\Bigr]
        \e^{\frac12\beta\Gamma_N G_{m,\epsilon}(0,0)}
        \,d\nu^{\otimes N},
\]
Since we have the straightforward decomposition
\[
  \prod_{j=1}^N A_{\epsilon j}^N(b)
    = \prod_{j=1}^N B_{\epsilon j}^N(b)
      + \sum_{k=1}^N\Bigl(\prod_{j=1}^{k-1} A_{\epsilon j}^N(b)\Bigr)
      \cdot D_{\epsilon k}^N(b)\cdot
      \Bigl(\prod_{j=k+1}^N B_{\epsilon j}^N(b)\Bigr),
\]
if we also set
\[
  \Lc(a,b)
    \eqdef\int\dots\int
      \e^{\frac12\beta(\Gamma_N-\Gamma_\infty)G_{m,\epsilon}(0,0)}
      \e^{\im a\cdot\gamma}
      \E_{U_{\beta\epsilon}}\Bigl[\prod_{j=1}^N B_{\epsilon j}^N(b)\Bigr]
      \,d\nu^{\otimes N},
\]
and
\[
  \begin{multlined}[.9\linewidth]
    \Ec(a,b)
      \eqdef\int\dots\int
      \e^{\frac12\beta(\Gamma_N-\Gamma_\infty)G_{m,\epsilon}(0,0)}
      \e^{\im a\cdot\gamma}\cdot\\
        \cdot\E_{U_{\beta\epsilon}}\Bigl[
        \sum_{k=1}^N\Bigl(\prod_{j=1}^{k-1} A_{\epsilon j}^N(b)\Bigr)
        D_{\epsilon k}^N(b)
        \Bigl(\prod_{j=k+1}^N B_{\epsilon j}^N(b)\Bigr)
        \Bigr]
        \,d\nu^{\otimes N},
  \end{multlined}
\]
we have that
\[
  \E_{\mu_{\beta,\epsilon}^N}\bigl[\e^{\im(a\cdot\gamma+b\cdot X)}\bigr]
    = \frac1{Z_{\beta\epsilon}^N}
      \e^{\frac12\beta\Gamma_\infty G_{m,\epsilon}(0,0)}
      \bigl(\Lc(a,b) + \Ec(a,b)\bigr)
\]
If in particular we take $a=b=0$, we obtain an analogous formula
for the partition function,
\[
  Z_{\beta\epsilon}^N
    = \e^{\frac12\beta\Gamma_\infty G_{m,\epsilon}(0,0)}
      \bigl(\Lc(0,0) + \Ec(0,0)\bigr),
\]
and in conclusion
\[
  \begin{aligned}
    \E_{\mu_{\beta,\epsilon}^N}\bigl[\e^{\im(a\cdot\gamma+b\cdot X)}\bigr]
      &= \frac{\Lc(a,b) + \Ec(a,b)}{\Lc(0,0) + \Ec(0,0)}\\
      &= \Bigl(\frac{\frac{\Lc(a,b)}{\Lc(0,0)}}{1 + \frac{\Ec(a,b)}{\Lc(0,0)}}
        + \frac{\frac{\Ec(a,b)}{\Lc(0,0)}}{1 + \frac{\Ec(a,b)}{\Lc(0,0)}}
        \Bigr).
  \end{aligned}
\]
It is sufficient now to prove that
\begin{equation}\label{e:fddclaim}
  \begin{gathered}
    \frac{\Lc(a,b)}{\Lc(0,0)}
      \longrightarrow\int\dots\int \e^{\im a\cdot\gamma+\im b\cdot x}
        \,d\ell^{\otimes n}\,d\nu^{\otimes n},\\
    \frac{\Ec(a,b)}{\Lc(0,0)}
      \longrightarrow 0,
  \end{gathered}
\end{equation}
as $N\uparrow\infty$, $\epsilon=\epsilon(N)\downarrow0$,
for all $a$ and $b$.

We first analyze $\Lc(a,b)/\Lc(0,0)$. If
$(U_{\beta,\epsilon,k})_{k\geq1}$
are the
components of $U_{\beta,\epsilon}$
with respect to the
eigenvectors $e_1,e_2,\dots$,
we notice that
$(U_{\beta,\epsilon,k})_{k\geq1}$
are independent centred Gaussian
random variables, and for each $k$,
$U_{\beta,\epsilon,k}$ has
variance $\beta g^\epsilon_k$,
where we have set for brevity
$g^\epsilon_k\eqdef \lambda_k^{-m/2}\e^{-\epsilon\lambda_k}$.
Therefore, by Plancherel,
\[
  \frac12\sum_{j=1}^N\|\tfrac{\gamma_j}{\sqrt N}U_{\beta,\epsilon}\|_{L^2(\ell)}^2
    = \frac12\Gamma_N\sum_{k=1}^\infty U_{\beta,\epsilon,k}^2,
\]
and by independence,
\[
  \begin{aligned}
    \E_{U_{\beta\epsilon}}\Bigl[\prod_{j=1}^N B_{\epsilon j}^N(b)\Bigr]
      &= \Bigl(\prod_{j=1}^n\int_\Torus\e^{\im b_j\cdot x_j}\,d\ell\Bigr)
        \E_{U_{\beta\epsilon}}\Bigl[\e^{-\frac12\Gamma_N
        \sum_{k=1}^\infty U_{\beta,\epsilon,k}^2}\Bigr]\\
      &= \Bigl(\prod_{j=1}^n\int_\Torus\e^{\im b_j\cdot x_j}\,d\ell\Bigr)
        \prod_{k=1}^\infty\E_{U_{\beta\epsilon}}\Bigl[
        \e^{-\frac12\Gamma_N U_{\beta,\epsilon,k}^2}\Bigr]\\
      &= \Bigl(\prod_{j=1}^n\int_\Torus\e^{\im b_j\cdot x_j}\,d\ell\Bigr)
        \prod_{k=1}^\infty\frac1{(1+\beta g^\epsilon_k\Gamma_N)^\frac12}
  \end{aligned}
\]
using elementary Gaussian integration.
Thus we have
\begin{equation}\label{e:elleab}
  \begin{multlined}[.8\linewidth]
  \Lc(a,b)
    = \Bigl(\prod_{j=1}^n\int_\Torus\e^{\im b_j\cdot x_j}\,d\ell\Bigr)
      \Bigl(\prod_{k=1}^\infty
        \frac1{(1+\beta g^\epsilon_k\Gamma_\infty)^\frac12}\Bigr)\cdot\\
      \cdot\int\dots\int
      \e^{\im a\cdot\gamma}F_N(\Gamma_N-\Gamma_\infty)
      \,d\nu^{\otimes N},
  \end{multlined}
\end{equation}
where
\begin{equation}\label{e:effeenne}
  F_N(X)
    = \e^{\frac12\beta X G_{m,\epsilon}(0,0)}
      \prod_{k=1}^\infty\Bigl(1 +
        \frac{\beta g^\epsilon_k}
        {1+\beta g^\epsilon_k\Gamma_\infty}X\Bigr)^{-\frac12}.
\end{equation}
If we prove that $F_N$ meets the assumptions of
Lemma~\ref{l:bernstein}, we immediately have
the first claim of \eqref{e:fddclaim}.
Indeed, set
\[
  c_k
    = \frac{\beta g^\epsilon_k}{1+\beta g^\epsilon_k\Gamma_\infty},
\]
then, by using the elementary inequality
$\log(1+x)\geq x-\frac12x^2$,
\[
  \begin{aligned}
    2\log F_N(x)
      &= \beta G_{m,\epsilon}(0,0)x
         - \sum_{k=1}^\infty\log(1+c_k x)\\
      &\leq \Bigl(\beta G_{m,\epsilon}(0,0) - \sum_{k=1}^\infty c_k\Bigr)x
        + \frac12 \Bigl(\sum_{k=1}^\infty c_k^2\Bigr)x^2\\
      &\leq \Bigl(\beta G_{m,\epsilon}(0,0) - \sum_{k=1}^\infty c_k\Bigr)x
        + \frac12 \Bigl(\sum_{k=1}^\infty c_k\Bigr)^2 x^2.
  \end{aligned}
\]
Since 
\[
  0
    \leq \sum_k c_k
    \leq \beta\sum_k g^\epsilon_k
    = \beta G_{m,\epsilon}(0,0),
  \]
both assumptions of the lemma hold if
there is $\alpha<\frac14$ such that
$G_{m,\epsilon}(0,0)\lesssim N^\alpha$.
On the other hand, it is elementary
to see that
\begin{equation}\label{e:gzero}
  G_{m,\epsilon}(0,0)
    = \sum_{k=1}^\infty g^\epsilon_k
    = \sum_{k=1} \lambda_k^{-\frac{m}2}\e^{-\epsilon\lambda_k}
    \approx\epsilon^{-\frac12(2-m)},
\end{equation}
since $\lambda_k\sim k$,
therefore our choice of $\epsilon=\epsilon(N)$ is sufficient
to ensure the assumptions of Lemma~\ref{l:bernstein}
for $F_N$. 

We turn to the analysis of $\Ec(a,b)/\Lc(0,0)$.
By their definition, we have that
$|A_{\epsilon j}^N(b)|\leq 1$ and
$|B_{\epsilon j}^N(b)|\leq 1$,
therefore,
\[
  \Bigl|\E_{U_{\beta\epsilon}}\Bigl[
      \sum_{k=1}^N\Bigl(\prod_{j=1}^{k-1} A_{\epsilon j}^N(b)\Bigr)
      D_{\epsilon k}^N(b)
      \Bigl(\prod_{j=k+1}^N B_{\epsilon j}^N(b)\Bigr)
      \Bigr]\Bigr|
    \leq \sum_{k=1}^N \E_{U_{\beta\epsilon}}[|D_{\epsilon k}^N(b)|],
\]
and it remains to estimate the expectations of the terms
$|D_{\epsilon j}^N(b)|$.

If $j\leq n$,
\[
  \begin{aligned}
    |D_{\epsilon j}^N(b)|
      &\leq \int_\Torus \bigl|
        \e^{\frac{\gamma_j}{\sqrt N} U_{\beta\epsilon}(x_j)}
        - \e^{-\frac{\gamma_j^2}{2N}\bigl\|U_{\beta\epsilon}\bigr\|_{L^2(\ell)}^2}
        \bigr|\,d\ell\\
      &\leq\frac{|\gamma_j|}{\sqrt N}\int_\Torus
        |U_{\beta\epsilon}(x_j)+\|U_{\beta\epsilon}\|_{L^2(\ell)}|\,d\ell\\
      &\lesssim\frac1{\sqrt N}\|U_{\beta\epsilon}\|_{L^2(\ell)},
  \end{aligned}
\]
where we have used the elementary inequalities
$|\e^{\im x}-1|\leq|x|$ and $1-\e^{-\frac12 y^2}\leq |y|$.
Thus for $j\leq n$,
\[
  \E_{U_{\beta\epsilon}}[|D_{\epsilon j}^N(b)|]
    \lesssim \frac1{\sqrt N}\E[\|U_{\beta\epsilon}\|_{L^2(\ell)}^2]^\frac12
    \lesssim \frac1{\sqrt N}\sqrt{G_{m,\epsilon}(0,0)},
\]
since $U_{\beta\epsilon}$ is a Gaussian random field with
covariance $\beta G_{m,\epsilon}$.

If on the other hand $j\geq n+1$, by Lemma~\ref{l:approx},
\[
  \begin{aligned}
    \E_{U_{\beta\epsilon}}[|D_{\epsilon j}^N(b)|]
      &\lesssim\frac1{N^{3/2}}\E_{U_{\beta\epsilon}}
        \bigl[\bigl\|\gamma_j U_{\beta\epsilon}
        \bigr\|_{L^3(\ell)}^3\bigr]\\
      &\leq \frac1{N^{3/2}}\bigl(\E_{U_{\beta\epsilon}}
        \bigl[\bigl\|\gamma_j U_{\beta\epsilon}
        \bigr\|_{L^4(\ell)}^4\bigr]\bigr)^{\frac34}\\
      &\lesssim\frac1{N^{3/2}}G_{m,\epsilon}(0,0)^{\frac32},
  \end{aligned}
\]
since
\begin{equation}\label{e:fourth}
  \begin{multlined}[.8\linewidth]
  \E_{U_{\beta,\epsilon}}[\|U_{\beta,\epsilon}\|_{L^4(\ell)}^4]
    = \int_\Torus \E[U_{\beta,\epsilon}(x)^4]\,d\ell =\\
    = \int_\Torus 3\beta^2 G_{m,\epsilon}(x,x)^2\,d\ell
    = 3\beta^2 G_{m,\epsilon}(0,0).
  \end{multlined}
\end{equation}

In conclusion
\[
  \begin{aligned}
    \Ec(a,b)
      &\leq \int\dots\int
        \e^{\frac12\beta(\Gamma_N-\Gamma_\infty)G_{m,\epsilon}(0,0)}
        \sum_{k=1}^N \E_{U_{\beta\epsilon}}[|D_{\epsilon k}^N(b)|]
        \,d\nu^{\otimes N}\\
      &\leq \Bigl(\frac{n}{\sqrt N}\sqrt{G_{m,\epsilon}(0,0)}
        + \frac{N-n}{N^{\frac32}}G_{m,\epsilon}(0,0)^{\frac32}\Bigr)
        \Ec_0\\
      &\lesssim\frac1{\sqrt N}(1+G_{m,\epsilon}(0,0)^{\frac32})\Ec_0,
  \end{aligned}
\]
where we have set for brevity
\begin{equation}\label{e:Ezero}
  \Ec_0
    \eqdef\int\dots\int
      \e^{\frac12\beta(\Gamma_N-\Gamma_\infty)G_{m,\epsilon}(0,0)}
      \,d\nu^{\otimes N}.
\end{equation}
It is easy to see that by Lemma~\ref{l:bernstein}, $\Ec_0\to1$.
Moreover, by our previous computations, see \eqref{e:elleab},
we can write $\Lc(0,0)$ as
\[
  \Lc(0,0)
    = \Bigl(\prod_{k=1}^\infty\frac1{1+\beta\Gamma_\infty g^\epsilon_k}\Bigr)^{\frac12}
    \Lc_0,
\]
with $\Lc_0\to1$. We have,
\begin{equation}\label{e:Gprod}
  \prod_{k=1}^\infty\frac1{1+\beta\Gamma_\infty g^\epsilon_k}
    = \e^{-\sum_k \log(1+\beta\Gamma_\infty g^\epsilon_k)}
    \geq \e^{-\sum_k \beta\Gamma_\infty g^\epsilon_k}
    = \e^{-\beta\Gamma_\infty G_{m,\epsilon}(0,0)},
\end{equation}
therefore
\[
  \frac{\Ec(a,b)}{\Lc(0,0)}
    \lesssim \frac1{\sqrt N}(1 + G_{m,\epsilon}(0,0)^{\frac32})
      \e^{\beta\Gamma_\infty G_{m,\epsilon}(0,0)}
      \frac{\Ec_0}{\Lc_0}.
\]
So it is sufficient to choose $\epsilon=\epsilon(N)$
so that
\begin{equation}\label{e:condition}
  \frac1{\sqrt N}(1 + G_{m,\epsilon}(0,0)^{\frac32})
      \e^{\beta\Gamma_\infty G_{m,\epsilon}(0,0)}
    \longrightarrow0
\end{equation}
Using \eqref{e:gzero}, we see immediately that
it suffices to choose
$\epsilon^{-\frac12(2-m)}\leq c\log N$, with $c$ small enough.
\subsection{Law of large numbers}\label{s:lln}

We turn to the proof of Theorem~\ref{t:lln}
on the weak law of large numbers
for point vortices. We will broadly
follow the same strategy of the previous
section.

First of all, we notice that it is sufficient
to prove convergence in law of $\eta_N$ to
$\ell\otimes\nu$. Moreover, in
view of Lemma~\ref{l:weak_conv}, it is
sufficient to prove convergence of
the characteristic functions over
test functions $\psi\in C(K_\nu\times\Torus)$.
Fix $\psi\in C(K_\nu\times\Torus)$,
then by using Lemma~\ref{l:gaussian_rep},
\[
  \begin{multlined}
    \E_{\mu_{\beta,\epsilon}^N}[\e^{\im\scalar{\psi,\eta_N}}] = \\
      = \frac1{Z_{\beta,\epsilon}^N}
        \int\ldots\int \E_{U_{\beta,\epsilon}^N}\bigl[
        \e^{\frac{\im}{\sqrt N}\sum_{j=1}^N\frac1{\sqrt N}\psi(\gamma_j,x_j)
          + \gamma_j U_{\beta,\epsilon}(x_j)}\bigr]
        \e^{\frac12\beta\Gamma_N G_{m,\epsilon}(0,0)}
        \,d\ell^{\otimes N}\,d\nu^{\otimes N},
  \end{multlined}
\]
where $\Gamma_N$ has been defined in \eqref{e:gammaN}.
We set now some notations to shorten and simplify our formulas.
Let $\ell(\psi)(\gamma)=\int\psi(\gamma,x)\,\ell(dx)$ and
$\phi=\psi-\ell(\psi)$. For a function $a\in C(K_\nu)$,
define
\begin{equation}\label{e:emmenne}
  M_N(a)
    = \frac1N\sum_{j=1}^N a(\gamma_j).
\end{equation}
Set moreover,
\[
  \begin{aligned}
    A_{\epsilon j}^N(\phi)
      &\eqdef\int_\Torus \e^{\frac\im{\sqrt N}
        \bigl(\frac1{\sqrt N}\phi(\gamma_j,x_j)
        +\gamma_j U_{\beta\epsilon}(x_j)\bigr)}\,d\ell,\\
    B_{\epsilon j}^N(\phi)
      &\eqdef \e^{-\frac1{2N}\bigl\|\frac1{\sqrt N}\phi(\gamma_j,\cdot)
        +\gamma_j U_{\beta\epsilon}\bigr\|_{L^2(\ell)}^2},\\
    D_{\epsilon j}^N(\phi)
      &\eqdef A_{\epsilon j}^N(\phi) - B_{\epsilon j}^N(\phi).
  \end{aligned}
\]
We have the decomposition
\begin{equation}\label{e:decompose}
    \prod_{j=1}^N A_{\epsilon j}^N(\phi)
      = \prod_{j=1}^N B_{\epsilon j}^N(\phi)
        + \sum_{k=1}^N\Bigl(\prod_{j=1}^{k-1} A_{\epsilon j}^N(\phi)\Bigr)
        \cdot D_{\epsilon j}^N(\phi)\cdot
        \Bigl(\prod_{j=k+1}^N B_{\epsilon j}^N(\phi)\Bigr).
\end{equation}
If we also set
\[
  \Lc(\psi)
    \eqdef\int\dots\int
      \e^{\frac12\beta(\Gamma_N-\Gamma_\infty)G_{m,\epsilon}(0,0)}
      \e^{\im M_N(\ell(\psi))}
      \E_{U_{\beta\epsilon}}\Bigl[\prod_{j=1}^N B_{\epsilon j}^N(\phi)\Bigr]
      \,d\nu^{\otimes N},
\]
and
\[
  \begin{multlined}[.9\linewidth]
    \Ec(\psi)
      \eqdef\int\dots\int
      \e^{\frac12\beta(\Gamma_N-\Gamma_\infty)G_{m,\epsilon}(0,0)}
        \e^{\im M_N(\ell(\psi))}\cdot\\
        \cdot\E_{U_{\beta\epsilon}}\Bigl[
          \sum_{k=1}^N\Bigl(\prod_{j=1}^{k-1} A_{\epsilon j}^N(\phi)\Bigr)
          D_{\epsilon k}^N(\phi)
          \Bigl(\prod_{j=k+1}^N B_{\epsilon j}^N(\phi)\Bigr)
        \Bigr]
        \,d\nu^{\otimes N},
  \end{multlined}
\]
we have that
\[
  \E_{\mu_{\beta,\epsilon}^N}[\e^{\im\scalar{\psi,\eta_N}}]
    = \frac1{Z_{\beta\epsilon}^N}
      \e^{\frac12\beta\Gamma_\infty G_{m,\epsilon}(0,0)}
      \bigl(\Lc(\psi) + \Ec(\psi)\bigr).
\]
A similar formula can be obtained for $Z_{\beta\epsilon}^N$,
therefore
\[
  \begin{aligned}
    \E_{\mu_{\beta,\epsilon}^N}[\e^{\im\scalar{\psi,\eta_N}}]
      &=\frac{\Lc(\psi) + \Ec(\psi)}{\Lc(0) + \Ec(0)}
  \end{aligned}
\]
It is sufficient now to prove that
\[
  \frac{\Lc(\psi)}{\Lc(0)}
    \longrightarrow\e^{\im\nu\otimes\ell(\psi)}
      \qquad\text{and}\qquad
  \frac{\Ec(\psi)}{\Lc(0)}
    \longrightarrow 0,
\]
as $N\uparrow\infty$, $\epsilon=\epsilon(N)\downarrow0$,
for all $\psi$.

We first analyze $\Lc(\psi)/\Lc(0)$. Let
$(U_{\beta,\epsilon,k})_{k\geq1}$ and
$(\phi_k)_{k\geq1}$ be the
component of $U_{\beta,\epsilon}$ and
$\phi$ with respect to the
eigenvectors $e_1,e_2,\dots$,
and we set again 
$g^\epsilon_k\eqdef\lambda_k^{-m/2}\e^{-\epsilon\lambda_k}$.
By independence and elementary Gaussian integration,
\[
  \begin{aligned}
    \E_{U_{\beta\epsilon}}\Bigl[\prod_{j=1}^N B_{\epsilon j}^N(\phi)\Bigr]
      &=\E_{U_{\beta\epsilon}}\Bigl[\e^{-\frac1{2N}\sum_{j=1}^N
          \|\frac1{\sqrt N}\phi(\gamma,\cdot)+\gamma_j U_{\beta,\epsilon}\|_{L^2(\ell)}^2}
        \Bigr]\\
      &=\e^{-\frac1{2N}M_N(\|\phi\|_{L^2(\ell)}^2)}
        \prod_{k=1}^\infty\E_{U_{\beta,\epsilon}}\Bigl[
        \e^{-\frac1{\sqrt N} M_N(\gamma\phi_k) U_{\beta,\epsilon,k}
          -\frac12\Gamma_N U_{\beta,\epsilon,k}^2}
        \Bigr]\\
      &=\e^{-\frac1{2N}M_N(\|\phi\|_{L^2(\ell)}^2)}
        \prod_{k=1}^\infty\Bigl(\frac1{(1+\beta\Gamma_N g^\epsilon_k)^{\frac12}}
        \e^{\frac1{2N}\frac{\beta M_N(\gamma\phi_k)^2 g^\epsilon_k}
          {1+\beta\Gamma_N g^\epsilon_k}}\Bigr).
  \end{aligned}
\]
Thus we have
\begin{equation}\label{e:elleaphi}
  \Lc(\psi)
    = \Bigl(\prod_{k=1}^\infty\frac1{(1+\beta\Gamma_\infty g^\epsilon_k)^{\frac12}}\Bigr)
      \e^{\im\nu\otimes\ell(\psi)}\Lc_0(\psi),
\end{equation}
where
\begin{equation}\label{e:ellezero}
  \begin{multlined}[.9\linewidth]
  \Lc_0(\psi)
    \eqdef
    \int\dots\int F_N(\Gamma_N-\Gamma_\infty)
      \e^{\im(M_N(\ell(\psi))-\nu\otimes\ell(\psi))}\cdot\\
      \cdot\e^{-\frac1{2N}M_N(\|\phi\|_{L^2(\ell)}^2)}
      \e^{\frac1{2N}\sum_{k=1}^\infty
        \frac{\beta M_N(\gamma\phi_k)^2 g^\epsilon_k}
        {1+\beta\Gamma_N g^\epsilon_k}}
      \,d\nu^{\otimes N},
  \end{multlined}
\end{equation}
and where $F_N$ has been defined in \eqref{e:effeenne}.
Since we look at the ratio $\Lc(\psi)/\Lc(0)$, it is sufficient to prove
that $\Lc_0(\psi)\to1$ as $N\uparrow\infty$ and $\epsilon=\epsilon(N)\downarrow0$,
for all $\psi$. In view of Lemma~\ref{l:bernstein}, we notice that
\[
  \e^{\im(M_N(\ell(\psi))-\nu\otimes\ell(\psi))}
    \longrightarrow 1,
\]
{a.\,s.} by the law of large numbers, and it is bounded. Likewise,
the same holds for
\[
  \e^{-\frac1{2N}M_N(\|\phi\|_{L^2(\ell)}^2)}
    \longrightarrow 1.
\]
Finally, since $M_N(\gamma\phi_k)^2\leq\Gamma_NM_N(\phi_k^2)$,
\begin{equation}\label{e:expolim}
  \sum_{k=1}^\infty
      \frac{\beta M_N(\gamma\phi_k)^2 g^\epsilon_k}
      {1+\beta\Gamma_N g^\epsilon_k}
    \leq \sum_{k=1}^\infty M_N(\phi_k^2)
    = M_N(\|\phi\|_{L^2(\ell)}^2),
\end{equation}
is bounded, we also have that
\[
  \e^{\frac1{2N}\sum_{k=1}^\infty
      \frac{\beta M_N(\gamma\phi_k)^2 g^\epsilon_k}
      {1+\beta\Gamma_N g^\epsilon_k}}
    \longrightarrow 1,
\]
and is bounded.
Since we have proved in the previous section
that $F_N$ meets the assumptions of
Lemma~\ref{l:bernstein}, we conclude
that $\Lc_0(\psi)\to1$.

We turn to the analysis of $\Ec(\psi)/\Lc(0)$.
By Lemma~\ref{l:approx} and formula \eqref{e:fourth},
\[
  \begin{aligned}
    \E_{U_{\beta\epsilon}}[|D_{\epsilon j}^N(\phi)|]
      &\lesssim\frac1{N^{3/2}}\E_{U_{\beta\epsilon}}\bigl[\bigl\|
        \tfrac1{\sqrt N}\phi(\gamma_j,\cdot) + \gamma_j U_{\beta\epsilon}
        \bigr\|_{L^3(\ell)}^3\bigr]\\
      &\leq \frac1{N^{3/2}}\bigl(\E_{U_{\beta\epsilon}}\bigl[\bigl\|
        \tfrac1{\sqrt N}\phi(\gamma_j,\cdot) + \gamma_j U_{\beta\epsilon}
        \bigr\|_{L^4(\ell)}^4\bigr]\bigr)^{\frac34}\\
      &\lesssim\frac1{N^{3/2}}(1 + G_{m,\epsilon}(0,0)^{\frac32}).
  \end{aligned}
\]
Therefore,
\[
  \begin{aligned}
    \E_{U_{\beta\epsilon}}\Bigl[
        \sum_{k=1}^N\Bigl(\prod_{j=1}^{k-1} A_{\epsilon j}^N(\phi)\Bigr)
          D_{\epsilon k}^N(\phi)
          \Bigl(\prod_{j=k+1}^N B_{\epsilon j}^N(\phi)\Bigr)
        \Bigr]
      &\leq \sum_{k=1}^N \E_{U_{\beta\epsilon}}[|D_{\epsilon k}^N(\phi)|]\\
      &\lesssim \frac1{\sqrt N}(1 + G_{m,\epsilon}(0,0)^{\frac32}),
  \end{aligned}
\]
so in conclusion
\[
  \Ec(\psi)
    \lesssim \frac1{\sqrt N}(1 + G_{m,\epsilon}(0,0)^{\frac32})\Ec_0,
\]
where $\Ec_0$ is defined as in \eqref{e:Ezero}, and $\Ec_0\to1$.
Using the expression of $\Lc(0)$ given by \eqref{e:elleaphi},
and formula~\eqref{e:Gprod}, we have
\[
  \frac{\Ec(\psi)}{\Lc(0)}
    \lesssim \frac1{\sqrt N}(1 + G_{m,\epsilon}(0,0)^{\frac32})
      \e^{\beta\Gamma_\infty G_{m,\epsilon}(0,0)}
      \frac{\Ec_0}{\Lc_0(0)}.
\]
As in formula~\eqref{e:condition}, by our assumption
on $\epsilon=\epsilon(N)$, the right-hand side converges
to $0$.
\subsection{Central limit theorem}\label{s:clt}

We finally turn to the proof of Theorem~\ref{t:clt}
on the fluctuations of point vortices.
Again, the strategy is the same of the
previous sections.
By Lemma~\ref{l:weak_conv}, it
suffices to prove convergence of
the characteristic functions over
test functions $\psi$, with
$\psi\in C^1(K_\nu\times\Torus)$,
namely to prove that
\[
  \E_{\mu_{\beta,\epsilon}^N}[\e^{\im\scalar{\psi,\zeta_N}}]
    \longrightarrow\e^{-\frac12\sigma_\infty(\psi)^2}
\]
To this end fix $\psi\in C^1(K_\nu\times\Torus)$,
and use the same notations of the previous section,
namely $\phi=\psi-\ell(\psi)$ and the operator $M_N$
defined in \eqref{e:emmenne}. Let $(\phi_k)_{k\geq1}$ and
$(G_{m,k})_{k\geq1}$ be the Fourier coefficients of $\phi$
and $G_m$ with respect to the basis of eigenvectors
$e_1,e_2,\dots$. It is an elementary computation that
\begin{equation}\label{e:sigmainfty}
  \sigma_\infty(\psi)^2
    = \nu(\ell(\psi)^2) - \nu\otimes\ell(\psi)
      + \|\phi\|_{L^2(\nu\otimes\ell)}^2
      -\beta\sum_{k=1}^\infty
        \frac{G_{m,k}\nu(\gamma\phi_k)^2}{1+\beta\Gamma_\infty G_{m,k}}.
\end{equation}

By using Lemma~\ref{l:gaussian_rep},
\[
  \begin{multlined}[.9\linewidth]
    \E_{\mu_{\beta,\epsilon}^N}[\e^{\im\scalar{\psi,\zeta_N}}]
      = \frac1{Z_{\beta,\epsilon}^N}\int\ldots\int
        \e^{\im\sqrt{N}(M_N(\ell(\psi))-\nu\otimes\ell(\psi))}
        \e^{\frac12\beta\Gamma_N G_{m,\epsilon}(0,0)}\cdot\\
        \cdot\E_{U_{\beta,\epsilon}^N}\bigl[
        \e^{\frac{\im}{\sqrt N}\sum_{j=1}^N \phi(\gamma_j,x_j)
          + \gamma_j U_{\beta,\epsilon}(x_j)}\bigr]
        \,d\ell^{\otimes N}\,d\nu^{\otimes N},
  \end{multlined}
\]
where $\Gamma_N$ is defined in \eqref{e:gammaN}.
With the positions
\[
  \begin{aligned}
    A_{\epsilon j}^N(\phi)
      &\eqdef\int_\Torus \e^{\frac\im{\sqrt N}
        \bigl(\phi(\gamma_j,x_j)
        +\gamma_j U_{\beta\epsilon}(x_j)\bigr)}\,d\ell,\\
    B_{\epsilon j}^N(\phi)
      &\eqdef \e^{-\frac1{2N}\bigl\|\phi(\gamma_j,\cdot)
        +\gamma_j U_{\beta\epsilon}\bigr\|_{L^2(\ell)}^2},\\
    D_{\epsilon j}^N(\phi)
      &\eqdef A_{\epsilon j}^N(\phi) - B_{\epsilon j}^N(\phi).
  \end{aligned}
\]
the decomposition~\eqref{e:decompose} still holds.
Set also
\[
  \begin{gathered}
  E_N(\psi)
    = \e^{\im\sqrt{N}(M_N(\ell(\psi))-\nu\otimes\ell(\psi))},\\
  \Lc(\psi)
    \eqdef\int\dots\int
      \e^{\frac12\beta(\Gamma_N-\Gamma_\infty)G_{m,\epsilon}(0,0)}
      E_N(\psi)
      \E_{U_{\beta\epsilon}}\Bigl[\prod_{j=1}^N B_{\epsilon j}^N(\phi)\Bigr]
      \,d\nu^{\otimes N},
  \end{gathered}
\]
and
\[
  \begin{multlined}[.9\linewidth]
    \Ec(\psi)
      \eqdef\int\dots\int
        \e^{\frac12\beta(\Gamma_N-\Gamma_\infty)G_{m,\epsilon}(0,0)}
        E_N(\psi)\cdot\\
        \cdot\E_{U_{\beta\epsilon}}\Bigl[
        \sum_{k=1}^N\Bigl(\prod_{j=1}^{k-1} A_{\epsilon j}^N(\phi)\Bigr)
        D_{\epsilon k}^N(\phi)
        \Bigl(\prod_{j=k+1}^N B_{\epsilon j}^N(\phi)\Bigr)
        \Bigr]
        \,d\nu^{\otimes N},
  \end{multlined}
\]
then, as in the previous sections,
\[
  \E_{\mu_{\beta,\epsilon}^N}[\e^{\im\scalar{\psi,\eta_N}}]
    =\frac{\Lc(\psi) + \Ec(\psi)}{\Lc(0) + \Ec(0)},
\]
and it is sufficient now to prove that
\[
  \frac{\Lc(\psi)}{\Lc(0)}
    \longrightarrow\e^{-\frac12\sigma_\infty(\psi)^2}
      \qquad\text{and}\qquad
  \frac{\Ec(\psi)}{\Lc(0)}
    \longrightarrow 0,
\]
as $N\uparrow\infty$, $\epsilon=\epsilon(N)\downarrow0$,
for all $\psi$.

We first prove the convergence
of the ratio $\Lc(\psi)/\Lc(0)$.
Let $(U_{\beta,\epsilon,k})_{k\geq1}$ and
$(\phi_k)_{k\geq1}$ be the
components of $U_{\beta,\epsilon}$ and
$\phi$ with respect to the
eigenvectors $e_1,e_2,\dots$,
and set again 
$g^\epsilon_k\eqdef\lambda_k^{-m/2}\e^{-\epsilon\lambda_k}$.
By Plancherel, independence, and elementary Gaussian
integration,
\[
  \begin{aligned}
    \E_{U_{\beta\epsilon}}\Bigl[\prod_{j=1}^N B_{\epsilon j}^N(\phi)\Bigr]
      &=\E_{U_{\beta\epsilon}}\Bigl[\e^{-\frac1{2N}\sum_{j=1}^N
        \|\phi(\gamma_j,\cdot)+\gamma_j U_{\beta,\epsilon}\|_{L^2(\ell)}^2}
        \Bigr]\\
      &=\e^{-\frac12 M_N(\|\phi\|_{L^2(\ell)}^2)}
        \prod_{k=1}^\infty\E_{U_{\beta,\epsilon}}\Bigl[
        \e^{-\frac12(\Gamma_N U_{\beta,\epsilon,k}^2
          + 2M_N(\gamma\phi_k) U_{\beta,\epsilon,k})}
        \Bigr]\\
      &=\e^{-\frac12 M_N(\|\phi\|_{L^2(\ell)}^2)}
        \prod_{k=1}^\infty\Bigl(\frac1{(1+\beta\Gamma_N g^\epsilon_k)^{\frac12}}
        \e^{\frac{\beta g^\epsilon_k M_N(\gamma\phi_k)^2}
          {2(1+\beta\Gamma_N g^\epsilon_k)}}\Bigr).
\end{aligned}
\]
Thus we have
\[
  \Lc(\psi)
    = \Bigl(\prod_{k=1}^\infty\frac1{\sqrt{1+\beta\Gamma_\infty g^\epsilon_k}}\Bigr)
      \Lc_0(\psi),
\]
with
\begin{equation}\label{e:ellezero2}
  \begin{multlined}[.8\linewidth]
    \Lc_0(\psi)
      \eqdef\int\dots\int
        F_N(\Gamma_N-\Gamma_\infty)
        E_N(\psi)\cdot\\
        \cdot\e^{-\frac12 M_N(\|\phi\|_{L^2(\ell)}^2)}
        \e^{\frac12\beta\sum_{k=1}^\infty
        \frac{g^\epsilon_k M_N(\gamma\phi_k)^2}{1+\beta\Gamma_N g^\epsilon_k}}
        \,d\nu^{\otimes N},
  \end{multlined}
\end{equation}
and where $F_N$ has been defined in \eqref{e:effeenne}.
As in the previous proofs, it suffices to prove
that $\Lc_0(\psi)\to\e^{-\frac12\sigma_\infty(\psi)^2}$
as $N\uparrow\infty$ and $\epsilon=\epsilon(N)\downarrow0$,
for all $\psi$.
Since we know already by the proof of Theorem~\ref{t:chaos}
(see Section~\ref{s:chaos}) that $F_N$ verifies the
assumptions of Lemma~\ref{l:bernstein}, it is sufficient
to prove convergence in expectation of the other terms
in $\Lc_0(\psi)$. First,
\[
  \e^{-\frac12 M_N(\|\phi\|_{L^2(\ell)}^2)}
    \longrightarrow\e^{-\frac12\|\phi\|_{L^2(\nu\otimes\ell)}^2},
\]
and 
\[
  \e^{\frac12\beta\sum_{k=1}^\infty
      \frac{g^\epsilon_k M_N(\gamma\phi_k)^2}{1+\beta\Gamma_N g^\epsilon_k}}
    \longrightarrow\e^{\frac12\beta\sum_{k=1}^\infty
      \frac{G_{m,k}\nu(\gamma\phi_k)^2}{1+\beta\Gamma_\infty G_{m,k}}}
\]
converge {a.\,s.} and in $L^1$ by the strong law of
large numbers. The first term is obviously bounded,
the second is bounded by the computations
in \eqref{e:expolim}. Here we can pass to the limit
also in the sum using the smoothness of $\phi$.
Finally, by the central limit theorem for
{i.\,i.\,d.} random variables,
\begin{equation}\label{miniclt}
  E_N(\psi)
    \longrightarrow\e^{-\frac12(\nu(\ell(\psi)^2)-\nu\otimes\ell(\psi)^2)}.
\end{equation}
By recalling the explicit form of $\sigma_\infty(\psi)$
given in \eqref{e:sigmainfty}, we conclude that
$\Lc_0(\psi)$ converges to $\e^{-\frac12\sigma_\infty(\psi)^2}$.

We turn to the analysis of $\Ec(\psi)/\Lc(0)$.
By Lemma~\ref{l:approx} and formula \eqref{e:fourth},
\[
  \begin{aligned}
    \E_{U_{\beta\epsilon}}[|D_{\epsilon j}^N(\phi)|]
      &\lesssim\frac1{N^{3/2}}\E_{U_{\beta\epsilon}}\bigl[\bigl\|
        \phi(\gamma_j,\cdot) + \gamma_j U_{\beta\epsilon}
        \bigr\|_{L^3(\ell)}^3\bigr]\\
      &\lesssim\frac1{N^{3/2}}(1 + G_{m,\epsilon}(0,0)^{\frac32}),
  \end{aligned}
\]
therefore, as in the previous sections,
\[
  \Ec(\psi)
    \lesssim \frac1{\sqrt N}(1 + G_{m,\epsilon}(0,0)^{\frac32})\Ec_0,
\]
with $\Ec_0\to1$. In conclusion,
\[
  \frac{\Ec(\psi)}{\Lc(0)}
    \lesssim \frac1{\sqrt N}(1 + G_{m,\epsilon}(0,0)^{\frac32})
      \e^{\beta\Gamma_\infty G_{m,\epsilon}(0,0)}
      \frac{\Ec_0}{\Lc_0(0)}.
\]
As in formula~\eqref{e:condition}, by our assumption
on $\epsilon=\epsilon(N)$, the right-hand side converges
to $0$.
\bibliographystyle{amsalpha}

\end{document}